\documentclass[AMA,STIX1COL]{my-WileyNJD-v2}

\articletype{Research Article}

\received{07-July-2018}
\revised{21-Jan-2019}
\accepted{24-02-2019}

\usepackage{amsfonts}
\usepackage{amsmath}
\usepackage{moreverb}
\usepackage{relsize}
\usepackage[overload]{empheq}

\begin{document}

\title{A collocation method of lines for two-sided space-fractional advection-diffusion equations with variable coefficients}

\author[1]{Mohammed K. Almoaeet}

\author[1]{Mostafa Shamsi*}

\author[1]{Hassan Khosravian-Arab}

\author[2]{Delfim F. M. Torres}

\authormark{M. K. Almoaeet, M. Shamsi, H. Khosravian-Arab and D. F. M. Torres}


\address[1]{\orgdiv{Department of Applied Mathematics, 
Faculty of Mathematics and Computer Science}, 
\orgname{Amirkabir University of Technology}, 
\orgaddress{No. 424 Hafez Avenue, 
\state{Tehran}, 
\country{Iran}}}

\address[2]{\orgdiv{Department of Mathematics, 
Center for Research and Development in Mathematics and Applications (CIDMA)}, 
\orgname{University of Aveiro}, 
\orgaddress{3810-193, \state{Aveiro}, \country{Portugal}}}

\corres{*Mostafa Shamsi, Department of Applied Mathematics, 
Faculty of Mathematics and Computer Science, 
Amirkabir University of Technology. \email{m\_shamsi@aut.ac.ir}}

\fundinginfoheadtext{Torres was partially supported by FCT 
through the R\&D Unit CIDMA, project UID/MAT/04106/2019.}

\JELinfo{35R11; 65M20; 65M70}


\abstract[Summary]{We present the Method Of Lines (MOL), 
which is based on the spectral collocation method, to solve 
space-fractional advection-diffusion equations (SFADEs) 
on a finite domain with variable coefficients. We focus 
on the cases in which the SFADEs consist of both left- 
and right-sided fractional derivatives. To do so, 
we begin by introducing a new set of basis functions 
with some interesting features. The MOL, together with 
the spectral collocation method based on the new basis 
functions, are successfully applied to the SFADEs. 
Finally, four numerical examples, including benchmark 
problems and a problem with discontinuous advection and 
diffusion coefficients, are provided to illustrate the 
efficiency and exponentially accuracy of the proposed method.}



\jnlcitation{\cname{%
\author{M. K. Almoaeet}, 
\author{M. Shamsi}, 
\author{H. Khosravian-Arab}, and 
\author{D. F. M. Torres},
(\cyear{2019}), 
\ctitle{A collocation method of lines for two-sided space-fractional advection-diffusion equations with variable coefficients}, 
\cjournal{Mathematical Methods in the Applied Sciences}.}}

\maketitle


\section{Introduction}
\label{Intr}

Fractional Calculus (FC) started with some speculations of Leibniz  
and Euler, and it has been developed progressively up until now. 
The initial fundamental works on the subject of FC 
are gathered in the books of Miller and Ross \cite{miller1993introduction},  
Kilbas et al\cite{MR2218073}, Hilfer \cite{MR1890104}, 
Podlubny \cite{MR1658022} and Kiryakova \cite{MR1265940}.
Without any doubt, in today's world, the concept of FC 
arises in many application's fields. In fact, the recent 
developments of FC have shown that a large number of 
complex systems can be accurately modelled by ordinary 
and partial differential equations with fractional derivatives.
These facts have created a new revolution of development 
in various fields of physics and engineering.
\cite{west2001fractional,Aparcana20181256,Bazhlekova2018}

In the last decade, fractional Advection-Diffusion Equations (ADEs) 
have received considerable attention in physics, mathematics and applications. 
A fractional ADE contains derivatives of fractional order in time,  
space or time-space. These three types of fractional ADEs 
have been used increasingly in the modelling of physical phenomena. 
Fractional ADEs are more proper for describing some phenomena such 
as anomalous diffusions, which arise in complex dynamics 
\cite{Yang2017276}, transport of passive tracers 
carried by fluid flow in porous mediums \cite{Maqbool20161,AliAbro2017}, 
velocity of particle motion in blood flow or underground water 
\cite{jppAhokposi2017}, etc.
Various methods have been developed for solving fractional ADEs. 
We refer to some of them. 
Zheng et al\cite{Yunying} used the finite element method 
to solve a space-fractional ADE with non-homogeneous initial-boundary conditions. 
Yang and Liu\cite{QYang} considered  a numerical method for ADEs 
with Riesz space fractional derivatives on a bounded domain. 
Ding and Zhang\cite{Heng} presented a numerical technique 
for solving ADEs with Riesz space fractional derivatives. 
A homotopy analysis method were used to solve ADM with 
Riesz space fractional derivatives by Ray and Sahoo\cite{Ray20152840}.
The Crank--Nicolson--Galerkin and backward Euler approximate schemes 
for two-dimensional space-fractional ADEs were investigated by Yanmin et al\cite{Yanmin}.
To solve  ADE with two-sided space-time fractional derivative,  
a high-accuracy Legendre tau method is developed 
and a convergence analysis provided by Bhrawy and Zaky \cite{Bhrawy2016}.
A meshless method for solving time-fractional ADEs 
is considered by Tayebi et al\cite{TAYEBI2017655}.
Zhu et al\cite{Zhu2017} presented a differential quadrature method 
for solving time- and space- fractional ADEs. 
In the very recent paper \cite{li2018fractional}, Li et al study 
fractional convection operators from physical and mathematical 
points of view and present a numerical scheme for space-fractional ADEs. 
As other numerical methods for fractional ADEs, we can refer to 
Zhai et al\cite{ZHAI2014138},
Bhrawy et al\cite{Bhrawy20162053} and
Li and Deng\cite{Li2017282}.
For some numerical methods for other equations, 
which are relevant to fractional ADEs, 
we can refer the interested readers to 
Meerschaert and Tadjeran\cite{Meerschaert2}, 
Dehghan and Abbaszadeh\cite{Dehghan20183476},
Feng et al\cite{Feng201552,Feng20171155},
Esmaeili\cite{Esmaeili20171838}, Zaky\cite{Zaky2018b} and references therein.

Here, the following two-sided space-fractional ADE 
with variable coefficients is considered:
\begin{subequations}
\label{main_prob}
\begin{align}\label{eq1}
\frac{{\partial u(x,t)}}{{\partial t}} 
+ \Big( {c_\alpha^+}(x,t){_{0}^{R}D_{x}^{\alpha}} 
+ {c_\alpha^ - }(x,t){_{x}^{R}D_{\ell}^{\alpha}}\Big)u(x,t) 
=\left({c_\beta^+}(x,t){_{0}^{R}D_{x}^{\beta}} 
+ {c_\beta^ - }(x,t){_{x}^{R}D_{\ell}^{\beta}} 
\right)u(x,t)+s(x,t),
\end{align}
together with the initial and Dirichlet boundary conditions
\begin{align}
\label{eq2}
u(x, 0) =
f(x) ,
\end{align}
\begin{align}
\label{eq3}
u(0, t) = u(\ell,t) = 0.
\end{align}
\end{subequations}
In Eq. \eqref{eq1}, the advection coefficients 
$ {c_\alpha^ +}(x,t)\ge 0$ and ${c_\alpha^ - }(x,t)\ge 0$,
and the diffusion coefficients ${c_\beta^ +}(x,t)\ge 0$ and ${c_\beta^ - }(x,t)\ge 0$,
are some given functions, and ${_{0}^{R}D_{x}^{\sigma}}$ and ${_{0}^{R}D_{x}^{\sigma}}$ 
are the left- and right-sided Riemann--Liouville fractional derivatives 
of order $\sigma$, respectively. We remind that
\begin{eqnarray*}
_{a}^{R}D_{x}^{\sigma} g(x)
&:=&\dfrac{1}{\Gamma(m-\sigma)
}\dfrac{d^{m}}{dx^{m}} \int_{a}^{x}{(x-t)}^{m-\sigma-1 } g(t) dt,\\
_{x}^{R}D_{b}^{\sigma} g(x) 
&:=&\dfrac{(-1)^{m}}{\Gamma(m-\sigma)}
\dfrac{d^{m}}{dx^{m}}\int_{x}^{b}{(t-x)}^{m-\sigma-1} g(t) dt,
\end{eqnarray*}
where $m$ is the integer number such that $m-1\leq \sigma<m $. 
It should be noted that the problem is considered on the domain $0 < x <\ell$ 
and $ 0 \leq t \leq T$. We consider the case $0 < \alpha \leq 1$ and $1<\beta\leq2$, 
where the parameters $\alpha$ and $\beta$ are the fractional orders (fractors) 
of the spatial variable. Moreover, the known function $s(x, t)$ is a source term. 

We use the Method Of Lines (MOL) to solve problem \eqref{main_prob}.  
The MOL is one of the best methods for solving time-dependent partial 
differential equations numerically. It proceeds in two steps. 
In the first step, the spatial variable(s) are discretized 
by using one from various possible techniques, such as spectral, finite element, 
finite difference or meshless methods. In fact, the main idea of this method is to convert 
the partial differential equation to a system of initial values problems (IVPs). 
In the second step, a time integration method, such as a Runge--Kutta method, 
is utilized to solve the resulting system of IVPs.
The MOL has been successfully applied to solve various types of problems 
in engineering and science.\cite{Dehghan-Camwa,Dehghan-num,MR3526036,siam,Muller20165027}

In this paper, a suitable set of basis functions for spatial discretization 
is constructed by using Jacobi polynomials. Moreover, 
the right and left space-fractional derivatives of the basis 
functions are derived in closed form, which help us to simplify 
the spatial discretization of the problem. Then, the unknown solution 
is expanded in terms of the basis functions, and by using a collocation
technique, the problem is reduced to an explicit system of IVPs. 
By solving this system of IVPs, an approximation for the solution
of the problem is obtained.

In Section~\ref{sec:2}, we briefly review some of the standard 
definitions of Jacobi polynomials. 
The main aim of this paper is presented in Section~\ref{Sec3}, 
where we construct and develop our MOL  
for solving problem \eqref{main_prob}.
In Section~\ref{Sec4}, some numerical examples are provided to show 
the efficiency of the proposed method. Some concluding  
remarks are provided at the end of the paper, 
in Section~\ref{sec:05:conc} of conclusion. 


\section{Preliminaries on Jacobi polynomials}
\label{sec:2}

One of the most practical classes of orthogonal polynomials 
on a closed interval $[-1,1]$ is the Jacobi polynomials, denoted by 
$ P_{j}^{(a,b)}(x)$, $j=0,1,\ldots$. It is worthy to note that the 
Chebyshev, Legendre and Gegenbauer polynomials, are some important 
special cases of the Jacobi polynomials.
The Jacobi polynomials have great flexibility to develop efficient numerical methods 
for solving a wide range of fractional differential models. Consequently, in the last decade,
the Jacobi polynomials have been widely used to solve fractional problems. 
\cite{esi1,ESMAEILI2011918,Bhrawy20161765,BHRAWY2016832,CSFJAHANSHAHI2017295,Kamali20183155,Izadkhah20181301,Zaky2018}
Our method uses the Jacobi polynomials too. Thus, in this section we briefly review 
the Jacobi polynomials, Jacobi quadrature rules, and a relevant theorem 
on the fractional derivatives of Jacobi polynomials. 

An efficient formula for evaluating the Jacobi polynomials  
$P_{j}^{(a,b)}(x)$, $a,b > -1$, $x\in [-1,1]$, is given 
by the following three-term recurrence equation:
\begin{eqnarray*}
&&P_{0}^{(a,b)}(x):=1, \quad P_{1}^{(a,b)}(x):=\frac{1}{2}[(a+b+2)x+a-b],\\
&&P_{j+1}^{(a,b)}(x):=(A_{j}^{a,b} x-B_{j}^{a,b})P_{j}^{(a,b)}(x) 
- C_{j}^{a,b}P_{j-1}^{(a,b)}(x), \quad j=1,2, \ldots ,m,
\end{eqnarray*}
where
\begin{equation}
\label{RecJac_1}
\begin{aligned}
&A_{j}^{a,b}:=\frac{(a+b+2j+1)(a+b+2j+2)}{2(j+1)(a+b+j+1)},\\
&B_{j}^{a,b}:=\frac{(b^2-a^2)(a+b+2j+1)}{2(j+1)(a+b+j+1)(a+b+2j)},\\
&C_{j}^{a,b}:=\frac{(a+j)(b+j)(a+b+2j+2)}{(j+1)(a+b+j+1)(a+b+2j)}.
\end{aligned}
\end{equation}
The Jacobi polynomials, with respect to the weight function
$w^{(a,b)}(x):=(1-x)^a(1+x)^b$, are orthogonal:
\begin{equation}
\label{ortho}
\int_{-1}^{1}P_{j}^{(a,b)}(x)P_{i}^{(a,b)}(x)w^{(a,b)}(x) dx
= \gamma_{j}^{a,b}\delta_{ij},
\end{equation}
where
\begin{equation}
\label{Orth_1}
\gamma _{j}^{a,b}:= \frac{2^{a+b+1}\Gamma(j+a+1)\Gamma(j+b+1)}{(2j+a+b+1)j! \Gamma(j+a+b+1)}.
\end{equation}
We remind the following two properties of the Jacobi polynomials, 
which are useful for our purposes:
\begin{eqnarray}
&&
P_{j}^{(a,b)}(-x)=(-1)^jP_{j}^{(b,a)}(x),\label{Prop_3}\\  
&& \frac{d^m}{dx^m}P_{j}^{(a,b)}(x)=d^{a,b}_{j,m}P_{j-m}^{(a+m,b+m)}(x),
\qquad j\geq m \in \mathbb{N},\label{JacP2}
\end{eqnarray}
where
\begin{equation*}
d^{a,b}_{j,m}:=\frac{\Gamma(j+m+a+b+1)}{2^{m}\Gamma(j+a+b+2)}.
\end{equation*}
The Jacobi polynomials satisfy the following important relation:
\begin{equation} 
\label{3h}
P_{k}^{(a, b-1)}\left(\frac{2x}{\ell}-1\right) = \frac{k + a + b}{2k + a + b}
P_{k}^{(a, b)}\left(\frac{2x}{\ell}-1\right) +  \frac{k + a}{2k + a + b}
P_{k-1}^{(a, b)}\left(\frac{2x}{\ell}-1\right).
\end{equation}
Since the Jacobi polynomials are orthogonal in $[-1,1]$, all zeros of 
$P_{m+1}^{(a,b)}(\tau)$ are simple and belong to the interval $(-1,1)$ \cite{Gautschi}. 
These zeros are called the Jacobi--Gauss nodes with parameters $a$ and $b$, 
which we denote by $\{\xi^{(a,b)}_i\}_{i=0}^m$. 
It is worth mentioning that there is no closed-form formula 
to obtain the Jacobi--Gauss nodes. However, 
We can use stable and accurate methods to determine them
\cite{Gautschi}. 
The Jacobi--Gauss quadrature rule with parameters $a$ and $b$ is based 
on the Jacobi--Gauss nodes $\{\xi^{(a,b)}_i\}_{i=0}^m$ and can be used for
approximating the integral of a function over 
the interval $[-1,1]$ with weight $(1-x)^a(1+x)^b$ as
\begin{equation*} 
\int_{-1}^{1} (1-x)^a(1+x)^b f(x) \text{d}x 
\simeq \sum_{i=0}^m \omega^{(a,b)}_i f(\xi^{(a,b)}_i),
\end{equation*}
where $\omega^{(a,b)}_i$, $i=0,\ldots,m$, are the Jacobi--Gauss 
quadrature weights, which are computed by
\[
w_i^{(a,b)} := \frac{2^{a+b+1} \Gamma(m+a+2)\Gamma(m+b+2)}{(m+1)!
\Gamma(m+a+b+2) (1-x^2_i)[(P_{m+1}^{(a,b)})'(x_i)]}.
\]
The $(n+1)$-points Jacobi--Gauss quadrature rule is exact whenever $f(x)$ 
is a polynomial of degree less or equal than $2n+1$, i.e., 
in the Jacobi--Gauss quadrature rule, the degree of exactness is $2n+1$.

The following theorem plays a key role in establishing our method.

\begin{theorem}[Zayernouri and Karniadakis\cite{Mohsen}]
\label{Zayer}
Let $-1<b-\alpha<b$ and $b>0$. Then, for $x\in[0,\ell]$, we have
\begin{equation*}
{}_{0}^{R}D_{x}^{\alpha}\left[x^{b}P_{k}^{(a,b)}\left(\frac{2x}{\ell}-1\right)\right]
=\frac{\Gamma(k+b+1)}{\Gamma(k+b-\alpha+1)}x^{b-\alpha}P_{k}^{(a
+\alpha,b-\alpha)}\left(\frac{2x}{\ell}-1\right).
\end{equation*}
\end{theorem}


\section{The proposed method of lines}
\label{Sec3}

The first step of our method consists to introduce appropriate 
basis polynomials (or functions) and expand the unknown solution in terms of them. 


\subsection{Modified Jacobi basis functions}

A good choice of the basis functions is crucial from the numerical point of view. 
To do so, we begin by presenting a new set of basis functions and then 
some useful properties for them are given. 

\begin{definition} 
\label{def:MJFs}
Let $a,b>-1$. The Modified Jacobi Functions (MJFs) on $[0,\ell]$ are denoted 
by $\mathcal{J}^{(a,b,\rho,\theta)}_k(x)$, $k=0,1,\ldots$, and are defined as
\begin{equation*}
\mathcal{J}^{(a,b,\rho,\theta)}_k(x):=x^\rho(\ell-x)^\theta\ 
P_k^{(a,b)}\left(\frac{2x}{\ell}-1\right),\ k=0,1,\ldots.
\end{equation*}
\end{definition} 

\begin{remark}
It is worth to point out that our new class of MJFs is a generalization 
of the following four classes of well-known basis functions:
\begin{itemize}
\item if $\rho=\theta=0$ in Definition~\ref{def:MJFs},
then the classical Jacobi polynomials \cite{shen} on $[0,\ell]$ are obtained;

\item MJFs for $\theta=0$, $\rho=b=\alpha$, and $a=\alpha-q>-1$ are reduced 
to the basis functions presented in Esmaeili and Shamsi\cite{esi1};

\item MJFs for $\rho=b$, $\theta=0$, and $\theta=a$, $\rho=0$, 
are reduced to the Jacobi-polyfractonomials (or generalized Jacobi functions) 
\cite{Mohsen,Sheng};

\item MJFs for $\theta=a$ and $\rho=b$ are reduced to the basis 
functions defined in Mao and Karniadakis\cite{MaoKar:18}.
\end{itemize} 
\end{remark}

In the following, we give some simple but important properties of MJFs.

\begin{property}[Integer derivative]
\label{Prop_1} 
If $k$ is a positive integer number and $\rho,\theta>k$, $k=0,1,\ldots$, then
\begin{equation*}
\frac{d^k}{dx^k}\mathcal{J}^{(a,b,\rho,\theta)}_k(x)\Big|_{x=0}
=\frac{d^k}{dx^k}\mathcal{J}^{(a,b,\rho,\theta)}_k(x)\Big|_{x=\ell}=0.
\end{equation*}  
\end{property}

For stable and accurate evaluation of MJFs, 
the following three-term recurrence relation is advised.  

\begin{property}[Three-term recurrence relation]
\label{Prop_2}  
The MJFs satisfy the following recurrence relation:
\begin{eqnarray*}
&&\mathcal{J}^{(a,b,\rho,\theta)}_0(x)=x^\rho(\ell-x)^\theta, 
\ \mathcal{J}^{(a,b,\rho,\theta)}_1(x)
=\frac{1}{2}\left[(a+b+2)\left(\frac{2x}{\ell}-1\right)
+a-b\right]\ x^\rho(\ell-x)^\theta,\\
&&\mathcal{J}^{(a,b,\rho,\theta)}_{j+1}(x)=\left(A_{j}^{a,b}
\left(\frac{2x}{\ell}-1\right)-B_{j}^{a,b}\right)
\mathcal{J}^{(a,b,\rho,\theta)}_{j}(x) 
-C_{j}^{a,b}\mathcal{J}^{(a,b,\rho,\theta)}_{j-1}(x), 
\quad j=1,2,\ldots,m,
\end{eqnarray*}
where $A_{j}^{a,b}$, $B_{j}^{a,b}$ and $C_{j}^{a,b}$ are defined 
in \eqref{RecJac_1}.
\end{property}

It is easy to verify that the MJFs are orthogonal with respect to the weight function. 

\begin{property}[Orthogonality]
The MJFs are orthogonal with respect to the weight function 
$w^{(a,b,\rho,\theta)}(x)$ as follows:
\begin{align*}
\int_{0}^{\ell}\mathcal{J}^{(a,b,\rho,\theta)}_{i}(x)
\mathcal{J}^{(a,b,\rho,\theta)}_{j}(x)w^{(a,b,\rho,\theta)}(x) dx
=\left(\frac{\ell}{2}\right)^{a+b+1} \gamma_{j}^{a,b}\delta_{ij},
\end{align*}
where $w^{(a,b,\rho,\theta)}(x):=x^{b-2\rho}(\ell-x)^{a-2\theta}$ 
and $\gamma_{j}^{a,b}$ is defined in \eqref{Orth_1}.
\end{property}


\subsection{Spectral collocation method for spatial discretization}

Now, we are in the position to discretize the spatial variable. 
For this purpose, we start to approximate the unknown function $u(x,t)$ 
in the problem \eqref{main_prob} by $\tilde u_n(x,t)$ as follows:
\begin{equation}
\label{appu}
u(x,t)\simeq \tilde u_n(x,t)=\sum_{k=0}^n a_k(t)\varphi_k(x),
\end{equation}  
where 
\begin{equation}
\label{eqw1}
\varphi_k(x):=\lambda_k\  \mathcal{J}^{(1,1,1,1)}_{k}(x)
=\lambda_k x(\ell-x)P^{(1,1)}_k\left(\frac {2x} \ell-1\right)
\end{equation}
and
\begin{equation}
\label{ccoef}
\lambda_k:=\frac 1 {\ell^3}{\frac { \left( k+2 \right)  \left( 2\,k+3 \right) }{k+1}}.
\end{equation}
It is interesting to point out that, in our collocation method, we choose 
$\varphi_k(x)$, $k=0,\ldots,n$, as the basis functions. 
We should mention 
that our motivation to use these basis functions is that they have some good features. 
First, they satisfy the boundary conditions \eqref{eq3}, thus 
$\tilde u_n(x,t)$ automatically satisfies the boundary conditions \eqref{eq3}.
Second, as we will show in the next Subsection~\ref{sec:3.3}, 
the left and also right fractional derivatives of these basis functions 
can be obtained in a closed form, which simplifies the discretization stage.

\begin{remark}
We note that, in order to improve the efficiency of the method, it is desirable that 
MJFs with $\rho,\,\theta\in (0,1)$ are considered as the basis function. However, 
by this consideration, we cannot derive both left and right derivatives in a closed 
form based on Jacobi polynomials. Accordingly, we develop our method with $\rho=\theta=1$. 
\end{remark}

Now, from \eqref{appu}, we can obtain the following approximations for 
$\tfrac {\partial}{\partial t}u$,
${}_{0}^{R}D_{x}^{\alpha}  u$ and
${}_{0}^{R}D_{x}^{\beta} u$:
\begin{eqnarray}
&&\frac {\partial}{\partial t}  u(x,t)
\simeq\frac {\partial}{\partial t} \tilde u_n(x,t)
= \sum_{k=0}^n\dot a_k(t)\varphi_k(x), \label{Dt}\\
&&{}_{0}^{R}D_{x}^{\alpha} u(x,t)\simeq {}_{0}^{R}D_{x}^{\alpha} \tilde u_n(x,t)
= \sum_{k=0}^n a_k(t)\psi_k^{\alpha+}(x),\quad 
{}_{0}^{R}D_{x}^{\beta}  u(x,t)
\simeq  {}_{0}^{R}D_{x}^{\beta} \tilde u_n(x,t)
= \sum_{k=0}^n a_k(t)\psi_k^{\beta+}(x),\label{Dleftx1}\\
&&{}_{x}^{R}D_{\ell}^{\alpha} u(x,t)
\simeq {}_{x}^{R}D_{\ell}^{\alpha} \tilde u_n(x,t)
= \sum_{k=0}^n a_k(t)\psi_k^{\alpha-}(x),
\quad {}_{x}^{R}D_{\ell}^{\beta}  u(x,t)
\simeq  {}_{x}^{R}D_{\ell}^{\beta} \tilde u_n(x,t)
= \sum_{k=0}^n a_k(t)\psi_k^{\beta-}(x),\label{Dleftx2} 
\end{eqnarray}
where 
\begin{eqnarray}
&&\psi_k^{\alpha+}(x):={}_{0}^{R}D_{x}^{\alpha}\varphi_k(x),
\qquad \psi_k^{\beta+}(x):={}_{0}^{R}D_{x}^{\beta}\varphi_k(x),\label{hbv1}\\
&&\psi_k^{\alpha-}(x):={}_{x}^{R}D_{\ell}^{\alpha}\varphi_k(x),
\qquad \psi_k^{\beta-}(x):={}_{x}^{R}D_{\ell}^{\beta}\varphi_k(x).\label{hbv2}
\end{eqnarray} 
Note that the computation of $\psi_k^{\alpha+}$, $\psi_k^{\beta+}$, $\psi_k^{\alpha-}$ 
and $\psi_k^{\beta-}$ are not straightforward. Thus, in the next Subsection~\ref{sec:3.3}, 
we present a novel method for evaluating them. 

By substituting the approximations \eqref{Dt}, \eqref{Dleftx1} and \eqref{Dleftx2} 
into Eq. \eqref{eq1}, we get
\begin{multline}
\label{ConvertedEq}
\sum_{k=0}^n\dot a_k(t)\varphi_k(x)
+ \left(c_\alpha^{+}(x,t)\sum_{k=0}^n a_k(t)\psi_k^{\alpha+}(x)
+c_\alpha^{-}(x,t)\sum_{k=0}^n (-1)^k a_k(t)\psi_k^{\alpha-}(x)\right)\\
= \left(c_\beta^{+}(x,t)\sum_{k=0}^n a_k(t)\psi_k^{\beta+}(x)
+c_\beta^{-}(x,t)\sum_{k=0}^n (-1)^k a_k(t)\psi_k^{\beta-}(x)\right)+s(x,t).
\end{multline}
Let $\hat\xi_i^{(1,1)}$, $i=0,\ldots,n$, be the Jacobi--Gauss nodes associated 
with the interval $[0,\ell]$, which are defined as 
$\hat\xi_i^{(1,1)}:=\frac \ell 2 (\xi_i^{(1,1)}+1)$. 
By collocating \eqref{ConvertedEq} at $x=\hat\xi_i^{(1,1)}$, $i=0,\dots,n$, 
we arrive at
\begin{multline*}
\sum_{k=0}^n\dot a_k(t)\varphi_k(\hat \xi_i^{(1,1)})
+ \left(c_\alpha^{+}(\hat \xi_i^{(1,1)},t)\sum_{k=0}^n a_k(t)
\psi_k^{\alpha+}(\hat \xi_i^{(1,1)})+c_\alpha^{-}(\hat \xi_i^{(1,1)},t)
\sum_{k=0}^n (-1)^k a_k(t)\psi_k^{\alpha-}(\hat \xi_i^{(1,1)})\right)\\
=\left(c_\beta^{+}(\hat \xi_i^{(1,1)},t)\sum_{k=0}^n a_k(t)\psi_k^{\beta+}(\hat \xi_i^{(1,1)})
+c_\beta^{-}(\hat \xi_i^{(1,1)},t)\sum_{k=0}^n (-1)^k a_k(t)\psi_k^{\beta-}(\hat \xi_i^{(1,1)})\right)
+s(\hat \xi_i^{(1,1)},t).\\
\end{multline*}
The above equations can be expressed in the following matrix form:
\begin{equation*}
\mathbf M\,\dot{\mathbf{a}}(t)+\mathbf C^{+}_\alpha(t) \big[ 
\mathlarger{\mathfrak D}^{+}_\alpha\ \mathbf {a}(t)\big] 
+\mathbf C^{-}_\alpha(t) \big[\mathlarger{\mathfrak D}^{-}_\alpha 
\ \mathbf {a}(t)\big]=\mathbf C^{+}_\beta(t) \big[ 
\mathlarger{\mathfrak D}^{+}_\beta\ \mathbf {a}(t)\big] 
+\mathbf C^{-}_\beta(t) \big[\mathlarger{\mathfrak D}^{-}_\beta 
\ \mathbf {a}(t)\big]+\mathbf{s}(t),
\end{equation*}
where
\[
\mathbf M:=\begin{bmatrix}
\varphi_0(\hat \xi_0^{(1,1)})&\varphi_0(\hat \xi_1^{(1,1)})
&\dots &\varphi_0(\hat \xi_n^{(1,1)})\\
\varphi_1(\hat \xi_0^{(1,1)})&\varphi_1(\hat \xi_1^{(1,1)})
&\dots &\varphi_1(\hat \xi_n^{(1,1)})\\
\vdots& \vdots& \ddots & \vdots\\
\varphi_n(\hat \xi_0^{(1,1)})&\varphi_n(\hat \xi_1^{(1,1)})
&\dots &\varphi_n(\hat \xi_n^{(1,1)})\\
\end{bmatrix},\quad \mathbf a(t)
:=\begin{bmatrix}
a_0(t)\\
a_1(t)\\
\vdots\\
a_n(t)
\end{bmatrix},
\quad
\mathbf s(t):=\begin{bmatrix}
s(\hat \xi_0^{(1,1)},t)\\
s(\hat \xi_1^{(1,1)},t)\\
\vdots\\
s(\hat \xi_n^{(1,1)},t)
\end{bmatrix},
\quad
\]
\[
\mathbf C_\alpha^{+}(t):=\text{diag}\left(c_\alpha^{+}(\hat \xi_0^{(1,1)},t),
\dots,c_\alpha^{+}(\hat \xi_n^{(1,1)},t)\right),\quad  \quad 
\mathbf C_\alpha^{-}(t):=\text{diag}\left(c_\alpha^{-}(\hat \xi_0^{(1,1)},t),
\dots,c_\alpha^{-}(\hat \xi_n^{(1,1)},t)\right),
\]
\[
\mathbf C_\beta^{+}(t):=\text{diag}\left(c_\beta^{+}(\hat \xi_0^{(1,1)},t),
\dots,c_\beta^{+}(\hat \xi_n^{(1,1)},t)\right),\quad  \quad \mathbf 
C_\beta^{-}(t):=\text{diag}\left(c_\beta^{-}(\hat \xi_0^{(1,1)},t),
\dots,c_\beta^{-}(\hat \xi_n^{(1,1)},t)\right),
\]
and
\[
\mathlarger{\mathfrak D}_\alpha^{+}
:=\begin{bmatrix}
\psi_0^{\alpha+}(\hat \xi_0^{(1,1)})&\psi_1^{\alpha+}(\hat \xi_0^{(1,1)})
&\dots &\psi_n^{\alpha+}(\hat \xi_0^{(1,1)})\\
\psi_0^{\alpha+}(\hat \xi_1^{(1,1)})&\psi_1^{\alpha+}(\hat \xi_1^{(1,1)})
&\dots &\psi_n^{\alpha+}(\hat \xi_1^{(1,1)})\\
\vdots& \vdots& \ddots & \vdots\\
\psi_0^{\alpha+}(\hat \xi_n^{(1,1)})&\psi_1^{\alpha+}(\hat \xi_n^{(1,1)})
&\dots &\psi_n^{\alpha+}(\hat \xi_n^{(1,1)})
\end{bmatrix},
\quad
\mathlarger{\mathfrak D}_\alpha^{-}
:=\begin{bmatrix}
\psi_0^{\alpha-}(\hat \xi_0^{(1,1)})&\psi_1^{\alpha-}(\hat \xi_0^{(1,1)})
&\dots &\psi_n^{\alpha-}(\hat \xi_0^{(1,1)})\\
\psi_0^{\alpha-}(\hat \xi_1^{(1,1)})&\psi_1^{\alpha-}(\hat \xi_1^{(1,1)})
&\dots &\psi_n^{\alpha-}(\hat \xi_1^{(1,1)})\\
\vdots& \vdots& \ddots & \vdots\\
\psi_0^{\alpha-}(\hat \xi_n^{(1,1)})&\psi_1^{\alpha-}(\hat \xi_n^{(1,1)})
&\dots &\psi_n^{\alpha-}(\hat \xi_n^{(1,1)})\\
\end{bmatrix}.
\]
By collocating the initial condition \eqref{eq2} at $x=\hat \xi_i^{(1,1)}$,
$i=0,\ldots,n$, we obtain that
\[
\sum_{k=0}^na_k(0)\varphi_k(\hat \xi_i^{(1,1)})
=f(\hat \xi_i^{(1,1)}),\quad i=0,\ldots,n.
\]
The above equations can be expressed in the following matrix form:
\[
\mathbf M\, \mathbf a(0)=\mathbf F,
\]
where  
$\mathbf F:=[f(\hat \xi_0^{(1,1)}),f(\hat \xi_1^{(1,1)}),
\ldots,f(\hat \xi_n^{(1,1)})]^T$ and 
$\mathbf a(0):=[a_0(0),a_1(0),a_2(0),\ldots,a_n(0)]^T$.

In summary, the main problem \eqref{main_prob} is 
reduced to the following system of ordinary differential 
equations with initial condition:
\begin{subequations}
\label{main-IVP-im}
\begin{align}[left = \empheqlbrace\,]
&\mathbf M\,\dot{\mathbf{a}}(t)=-\mathbf C^{+}_\alpha(t) 
\big[ \mathlarger{\mathfrak D}^{+}_\alpha\ \mathbf {a}(t)\big] 
-\mathbf C^{-}_\alpha(t) \big[\mathlarger{\mathfrak D}^{-}_\alpha 
\ \mathbf {a}(t)\big]+\mathbf C^{+}_\beta(t) \big[ 
\mathlarger{\mathfrak D}^{+}_\beta\ \mathbf {a}(t)\big] 
+\mathbf C^{-}_\beta(t) \big[\mathlarger{\mathfrak D}^{-}_\beta 
\ \mathbf {a}(t)\big]+\mathbf{s}(t) ,\\
&\mathbf M \mathbf a(0)=\,\mathbf F.
\end{align}
\end{subequations}

If we compute the inverse of the mass matrix $\mathbf M$, then we can derive 
an explicit form of the above IVP. In the next theorem, we present the explicit 
form of the inverse of the mass matrix $\mathbf M$.

\begin{theorem}
\label{InvMas}
The mass matrix $\mathbf M$ is nonsingular and the $(i,j)$th entry 
of it can be obtained explicitly as
\begin{equation*}
\left(\mathbf M ^{-1}\right)_{i,j}
=\frac{\ell\, w_i^{(1,1)}}{2(1-\xi_i^{(1,1)})(1+ \xi_i^{(1,1)})} 
P_{j}^{(1,1)}\left(\xi_i^{(1,1)}\right),
\quad i,j=0,1,\ldots,n,
\end{equation*}
where $w_i^{(1,1)},\ i=0,1,2,\ldots,n$, are the Gauss--Jacobi weights.
\end{theorem}

\begin{proof}
To prove the theorem, it is sufficient to show that the inner product of the 
$i$th raw of $\mathbf M$ and the $j$th column of $\mathbf M^{-1}$ 
is equal to $\delta_{ij}$, i.e.,
\begin{equation}
\label{HOKM1}
\sum_{k=0}^n \varphi_i(\hat \xi_k^{(1,1)}) 
\frac{\ell\, w_k^{(1,1)}}{2(1-\xi_k^{(1,1)})(1+ \xi_k^{(1,1)})} 
P_{j}^{(1,1)}\left(\xi_k^{(1,1)}\right)=\delta_{ij},
\quad i,j=0,\ldots,n.
\end{equation}
To prove the above equations, 
we begin by using the orthogonality property \eqref{ortho}, 
which allows us to write that
\[
\int_{-1}^{1}P_{i}^{(1,1)}(x)P_{j}^{(1,1)}(x)(1-x)(1+x) dx
= \frac{2^{3}(i+1)}{(2i+3)(i+2)}
\delta_{ij},\quad i,j=0,\ldots,n.
\]
By noting that $P_{i}^{(1,1)}(x)P_{j}^{(1,1)}(x)$ is a polynomial of degree at most $2n$, 
we conclude that the integral in the above equation can be computed exactly 
by using the Jacobi--Gauss quadrature rule based on the nodes 
$\{\xi_j\}_{j=0}^n$. In this way, from the above equation, we get
\[
\sum_{k=0}^{n} w_k^{(1,1)}P_{i}^{(1,1)}(\xi_k^{(1,1)})P_{j}^{(1,1)}(\xi_k^{(1,1)})
= \frac{2^{3}(i+1)}{(2i+3)(i+2)}
\delta_{ij},\quad i,j=0,\ldots,n.
\]
The above equations can be reformulated as
\[
\sum_{k=0}^{n} \frac {\ell^3 w_k^{(1,1)}} {2^3\hat 
\xi_k^{(1,1)}(\ell-\hat \xi_k^{(1,1)})}\frac{(2i+3)(i+2)}{\ell^3(i+1)} 
\hat \xi_k^{(1,1)}(\ell-\hat \xi_k^{(1,1)})P_{i}^{(1,1)}(\xi_k^{(1,1)}) 
P_{j}^{(1,1)}(\xi_k^{(1,1)})= \delta_{ij}.
\]
Using Eqs. \eqref{eqw1}--\eqref{ccoef} and by noting that 
$\xi_k^{(1,1)} = \frac{2\hat \xi_k^{(1,1)}}\ell-1$ and 
$\hat \xi_k^{(1,1)}(\ell-\hat \xi_k^{(1,1)})
=\tfrac {\ell^2}{2^2}(1+\xi_k^{(1,1)})(1-\xi_k^{(1,1)})$, 
we arrive at Eq. \eqref{HOKM1} and thus the proof of the theorem is completed.
\end{proof}

By using $\mathbf M ^{-1}$, we can convert the implicit IVP 
\eqref{main-IVP-im} into the following explicit IVP:
\begin{subequations}
\label{main-IVP}
\begin{align}[left = \empheqlbrace\,]
&\dot{\mathbf{a}}(t)=\mathbf M^{-1}\left[-\mathbf C^{+}_\alpha(t)  
\mathlarger{\mathfrak D}^{+}_\alpha -\mathbf C^{-}_\alpha(t) 
\mathlarger{\mathfrak D}^{-}_\alpha +\mathbf C^{+}_\beta(t) 
\mathlarger{\mathfrak D}^{+}_\beta +\mathbf C^{-}_\beta(t) 
\mathlarger{\mathfrak D}^{-}_\beta\right] 
\mathbf {a}(t)+\mathbf M^{-1}\mathbf{s}(t) ,\\
& \mathbf a(0)=\mathbf M^{-1}\,\mathbf F.
\end{align}
\end{subequations}
By solving the above system of IVPs, with a well-developed time-marching method 
\cite{doi:10.1137/1037026}, the unknown vector function $\mathbf a(t)$ 
is obtained numerically. Then, by using \eqref{appu}, an approximation of $u(x,t)$ is obtained.


\subsection{On the computation of the left and right fractional derivatives of the basis functions}
\label{sec:3.3}

As we have seen, the fractional differentiation matrices 
${{\mathbf{\mathlarger{\mathfrak D}}}}_+^{\alpha+}$, 
${{\mathbf{\mathlarger{\mathfrak D}}}}_+^{\alpha-}$,
${{\mathbf{\mathlarger{\mathfrak D}}}}_+^{\beta+}$,
and
${{\mathbf{\mathlarger{\mathfrak D}}}}_+^{\beta-}$
have a key role in our  method. These matrices simplify 
the spatial discretization process by replacing the
left and right fractional differentiations with matrix-vector products.
To derive these matrices, we can use Eqs. \eqref{hbv1} and \eqref{hbv2}. 
However, before it, we need to obtain the functions 
$\psi_k^{\alpha+}$,
$\psi_k^{\alpha-}$,
$\psi_k^{\beta+}$
and
$\psi_k^{\beta-}$.
At first glance, it seems that
these functions can be easily obtained,
by taking the analytical left/right fractional derivative of $\varphi_k(x)$.
However, taking the analytical fractional derivative, especially for large $n$, 
is sophisticated. In this respect, in what follows, we present some theorems, 
which play key roles in obtaining the mentioned functions.

\begin{theorem}
\label{RED}
Let $\varphi_k(x)$, $k=0,1,\ldots$, be defined in \eqref{eqw1} and 
$\psi_k^{\sigma+}(x):={}_{0}^{R}D_{x}^{\sigma}\varphi_k(x)$. 
Then, for the cases $0<\sigma<1$ or $1<\sigma<2$, we have
\[
\psi_0^{\sigma+}(x)=\frac {6\Gamma(2)}{\ell^2\Gamma(2-\sigma)}x^{1-\sigma}
-\frac {6\Gamma(3)}{\ell^3\Gamma(3-\sigma)}x^{2-\sigma},
\]
and, for $k\ge 1$,
\[
\psi_k^{\sigma+}(x)={x}^{1-\sigma}\left({\tfrac { \left( 2\,k+3 \right) 
\Gamma  \left( k+3 \right) }{ \ell^2\left( k+1\right) 
\Gamma\left( k+2-\sigma \right)}}\hat P^{(1+\sigma,1-\sigma)}_k \left(x \right)
-{\tfrac { \left( k+2\right) \Gamma  \left( k+4 \right)}{ \ell^3\left( k+1 \right) 
\Gamma\left( k+3-\sigma \right)}}{x}\hat P^{(1+\sigma,2-\sigma)}_k \left(x\right)
-{\tfrac {\Gamma  \left( k+3 \right)  }{\ell^3\Gamma\left( k+2-\sigma\right)}}{x}
\hat P^{(1+\sigma,2-\sigma)}_{k-1}\left(x \right)\right).
\]
When $\sigma=2$, we get
\[
\psi_0^{2+}(x)=\frac {-12}{\ell^3},
\]
\[
\psi_k^{2+}(x)=\frac{\lambda_k}{\ell^2}\left((k+4)(k+3)x(\ell-x)
\hat P^{(3,3)}_{k-2}\left(x \right)+2\ell(k+3)(\ell-2x)
\hat P^{(2,2)}_{k-1}\left(x \right)-2\ell^2
\hat P^{(1,1)}_{k}\left(x \right)\right),
\]
where $\hat P^{(a,b)}_{k}\left(x \right):=\displaystyle 
P^{(a,b)}_{k}\left(\frac{2x}{\ell}-1 \right)$.
\end{theorem}

\begin{proof} 
The evaluation of $\psi_0(x)$ is trivial. 
For $k\ge 1$, we have
\begin{equation*} 
{\varphi_{k}}(x) 
= \lambda_{k} \left(\ell xP_{k}^{(1,1)}\left(\frac{2x}{\ell}-1\right) 
- x^2 P_{k}^{(1,1)}\left(\frac{2x}{\ell}-1\right)\right),
\end{equation*}
and thus
\begin{equation} 
\label{3f}
{}_{0}^RD_{x}^{\sigma}{\mathbf {\varphi_{k}}}(x) 
= \lambda_{k}\left( \;{}_{0}^RD_{x}^{\sigma}\left[
\ell xP_{k}^{(1,1)}\left(\frac{2x}{\ell}-1\right)\right] 
- \;{}_{0}^RD_{x}^{\sigma}\left[x^2 P_{k}^{(1,1)}
\left(\frac{2x}{\ell}-1\right)\right)\right].
\end{equation}
Using Theorem~\ref{Zayer}, one has
\begin{equation} 
\label{3g}
\;{}_{0}^RD_{x}^{\sigma}\left[\ell\, \lambda_k\, x\, 
P_{k}^{(1,1)}\left(\frac{2x}{\ell}-1\right)\right] 
=  \frac{(2k + 3)\Gamma(k + 3)}{\ell^2(k + 1)
\Gamma(k + 2 - \sigma)}x^{1 - \sigma} P_{k}^{(1 
+ \sigma, 1 - \sigma)}\left(\frac{2x}{\ell}-1\right).
\end{equation} 
Moreover, using the property \eqref{3h}, 
with $a = 1$ and $b = 2$, we get
\begin{equation} 
\label{3i}
P_{k}^{(1, 1)}\left(\frac{2x}{\ell}-1\right) 
= \frac{k + 3}{2k + 3}P_{k}^{(1, 2)}\left(\frac{2x}{\ell}-1\right) 
+  \frac{k + 1}{2k + 3}P_{k-1}^{(1, 2)}\left(\frac{2x}{\ell}-1\right).
\end{equation}
Multiplying both sides of \eqref{3i} by $\lambda_kx^2$, 
and applying the left-side fractional derivative, we arrive at
\begin{equation*} 
\;{}_{0}D_{x}^{\sigma}\Big[\lambda_k x^2 P_{k}^{(1,1)}\left(
\frac{2x}{\ell}-1\right)\Big] = \lambda_{k} \;{}_{0}D_{t}^{\sigma} 
\left[x^2\left(\frac{k + 3}{2k + 3}P_{k}^{(1, 2)}\left(
\frac{2x}{\ell}-1\right)+ \frac{k + 1}{2k + 3}
P_{k-1}^{(1, 2)}\left(\frac{2x}{\ell}-1\right)\right) \right].
\end{equation*}
Using Theorem~\ref{Zayer} once again, we conclude that
\begin{multline} 
\label{3k}
\;{}_{0}D_{x}^{\sigma}\Big[\lambda_k x^2 P_{k}^{(1,1)}\left(
\frac{2x}{\ell}-1\right)\Big] 
= \frac{(k + 2)\Gamma(k + 4)}{
\ell^3(k + 1)\Gamma(k + 3 - \sigma)}x^{2-\sigma} 
P_{k}^{(1 + \sigma, 2 - \sigma)}\left(\frac{2x}{\ell}-1\right)\\
+ \frac{\Gamma(k + 3)}{\ell^3\Gamma(k + 2 - \sigma)}
x^{2-\sigma} P_{k-1}^{(1 + \sigma, 2 - \sigma)}\left(
\frac{2x}{\ell}-1\right).
\end{multline}
The proof of this theorem is concluded by plugging Eqs. \eqref{3g} 
and \eqref{3k} into relation \eqref{3f}. When $\sigma=2$, 
the use of property \eqref{JacP2} completes the proof.
\end{proof}

The next theorem is an important result, which has a key role 
for us to obtain a closed form for the right Riemann--Liouville 
fractional derivatives of the considered basis functions.

\begin{theorem}
\label{Main_1}
Let $\sigma>0$ and $m-1\leq \sigma<m$. Then
\begin{equation}
\psi_k^{\sigma-}(x):={}_{x}^{R}D_{\ell}^{\sigma} 
\varphi_{k}(x) =(-1)^m \psi_k^{\sigma+}(\ell-x).
\end{equation} 
\end{theorem}

\begin{proof}
We have
\[
\psi_k^{\sigma-}(x):=_{x}^{R}D_{\ell}^{\sigma} \varphi_{k}(x) 
=\dfrac{(-1)^{m}}{\Gamma(m-\sigma) }
\dfrac{d^{m}}{dx^{m}}\int_{x}^{\ell}{(t-x)}^{m-\sigma-1} 
\varphi_k(t) dt.
\]
By applying the change of variable 
$t\rightarrow \ell-\tau$, we get
\[
\psi_k^{\sigma-}(x)=-\dfrac{(-1)^{m}}{\Gamma(m-\sigma)}
\dfrac{d^{m}}{dx^{m}} 
\int_{\ell-x}^{0}{(\ell-\tau-x)}^{m-\sigma-1} 
\varphi_k(\ell-\tau) d\tau
\]
and, by noting that $\varphi_k(\ell-\tau)=\varphi_k(\tau)$, 
we conclude that
\begin{eqnarray*}
\psi_k^{\sigma-}(x)&=&(-1)^{m} \dfrac{1}{\Gamma(m-\sigma) }
\dfrac{d^{m}}{dx^{m}} \int_{0}^{\ell-x}{(\ell-x-\tau)}^{m-\sigma-1} 
\varphi_k(\tau) d\tau\\
&=& (-1)^m \Big[ {}_{0}^{R}D_{x}^{\sigma} 
\varphi_{k}(x)\Big]_{x\rightarrow \ell-x}\\
&=& (-1)^m \psi_k^{\sigma+}(\ell-x).
\end{eqnarray*}
The proof is complete.
\end{proof}

Theorem~\ref{Main_1} states that to evaluate the right-sided fractional 
derivative of order $\sigma$ of $\varphi_{k}(x)$, we can use 
the left-sided fractional derivative of order $\sigma$ 
of $\varphi_k(x)$, then substitute $x$ by 
$\ell-x$ and, finally, multiply it by $(-1)^m$. 
This property is related with the concept of duality
as introduced by Caputo and Torres \cite{MyID:307}.

In summary, by the explicit formulas in Theorems~\ref{RED} and \ref{Main_1},  
for functions $\psi_k^{\sigma+}$ and $\psi_k^{\sigma-}$, 
we can simply compute the elements of the matrices 
${{\mathbf{\mathlarger{\mathfrak D}}}}_+^{\alpha+}$, 
${{\mathbf{\mathlarger{\mathfrak D}}}}_+^{\alpha-}$,
${{\mathbf{\mathlarger{\mathfrak D}}}}_+^{\beta+}$,
and
${{\mathbf{\mathlarger{\mathfrak D}}}}_+^{\beta-}$.


\section{Numerical experiments}
\label{Sec4}

In this section we illustrate the proposed method of lines  
by using concrete numerical experiments. We have implemented our method  
using \textsc{Matlab} on a 3.5 GHz Core i7 PC 
with 8 GB of RAM.
Moreover, to solve IVP \eqref{main-IVP}, the solver \texttt{ode45} was used. 
This solver is based on an explicit Dormand--Prince Runge--Kutta (4, 5) 
formula \cite{DORMAND198019,MR1985643}. By using a suitable error estimation, 
the solver generates an adaptive mesh $t_0,\dots,t_k$ and obtains the values 
of the solution of IVP at this mesh. This formula is a single-step solver, i.e., 
in computing $\mathbf a(t_k)$, it needs only the solution at the immediately 
preceding time point, $\mathbf a(t_{k-1})$.
In \texttt{ode45}, we can adjust the accuracy of the solution 
through the input parameters \texttt{AbsTol} and \texttt{RelTol}. 
In our numerical experiments, we set \texttt{AbsTol}=$10^{-14}$ 
and \texttt{RelTol}=$10^{-12}$.

We consider four different examples. For the problems with 
a known exact solution, we assess the accuracy of our method 
by reporting the following $l_2$ and $l_\infty$ errors:
\begin{eqnarray*}
E_2(n)&:=&\sqrt{\frac 1 {100n}\sum_{i=0}^n\sum_{j=0}^{100}\left(
u(\hat \xi_i, j\tfrac T {100})
-\tilde u_n(\hat \xi_i, j\tfrac T {100})\right)^2},\\
E_\infty(n)&:=&\max\left\{\mid u(\hat \xi_i, j\tfrac T {100}) 
- \tilde u_n(\hat \xi_i, j\tfrac T {100})\mid : 
i=0,\ldots,n,\ j=0,\ldots,100\right\}.
\end{eqnarray*}


\subsection*{Example 1}

For the first example, we consider problem \eqref{main_prob} 
with the following data borrowed from Meerschaert and Tadjeran\cite{approximations}:
\begin{align*}
& \alpha:=1.8,\quad \ell:=2,
\quad T:=5,\quad f(x):=4\, \left( 2-x \right) ^{2}{x}^{2},\\
& {c_\alpha^+ }(x,t)={c_\alpha^+ }(x,t):=0,\quad {c_\beta^+ }(x,t)
:=\Gamma (1.2){x^{1.8}},\quad {c_\beta^- }(x,t):=\Gamma (1.2){(2 - x)^{1.8}},\\
& s(x,t):=  -\tfrac 4 {11}\,{{\rm e}^{-t}} \left( 211\,{x}^{4}-844\,{x}^{3}+1300\,{x}^{2}
- 912\,x+192 \right).
\end{align*}
In this problem, ehe exact solution is
\begin{equation*}
u_{\text{ex}}(x,t):=4\,{{\rm e}^{-t}}{x}^{2} 
\left( 2-x \right) ^{2}.
\end{equation*}
By applying the proposed MOL with $n=4$, the obtained approximated 
solution and the error function are plotted in Fig.~\ref{fig:fig1ex1}. 
Moreover, the errors $E_2(n)$ and $E_\infty(n)$ for various values 
of $n$, together with their CPU times, are shown in Fig.~\ref{fig:fig2ex1}.
As plotted in Fig.~\ref{fig:fig2ex1} (left), the errors $E_2(n)$ and $E_\infty(n)$
decrease very rapidly until, in $n=4$, they take over about $1e-13$. 
It is noted that, for $n>4$, the round-off errors on computer prevent 
any further improvement. However, it is worthwhile to note that, 
for $n>4$, the errors remain about $1e-13$, which shows that 
the presented MOL is numerically stable.
\begin{figure}
\centering
\includegraphics[width=\linewidth]{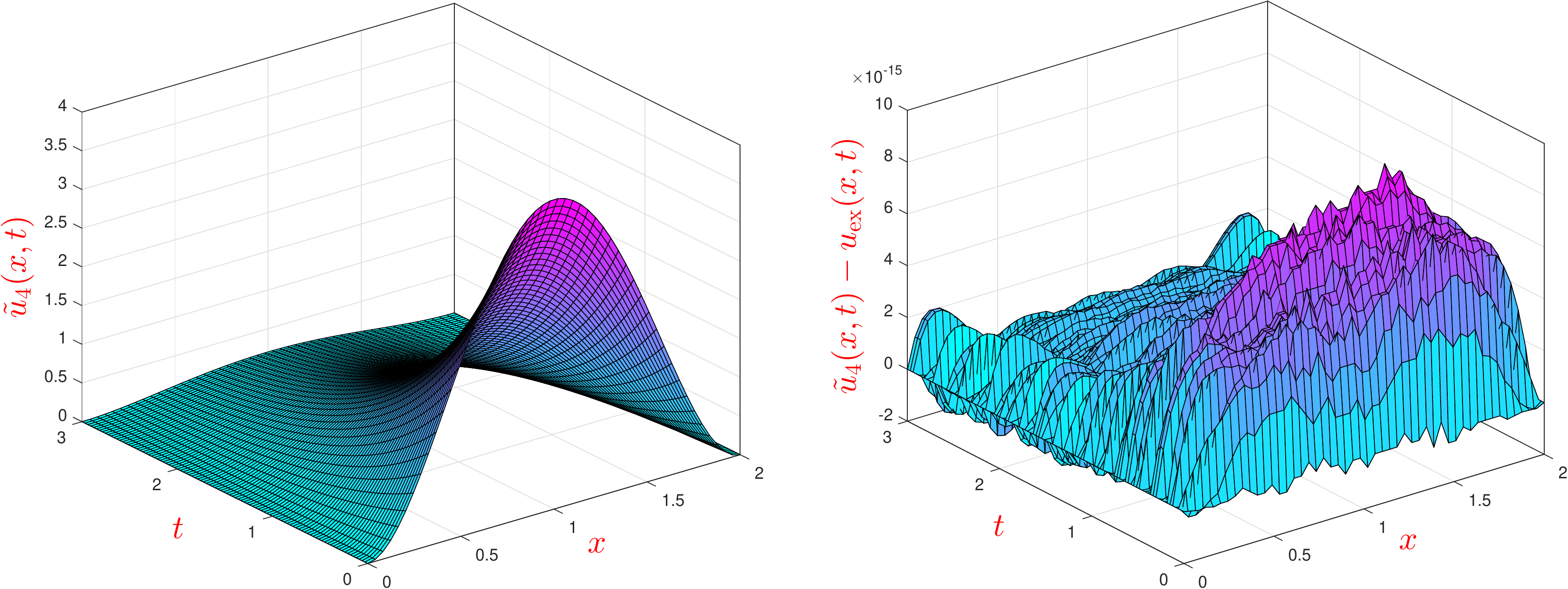}
\caption{The approximate solution for Example~1 (left) 
together with the error function (right) 
obtained by the proposed method with $n=4$.}
\label{fig:fig1ex1}
\end{figure}
\begin{figure}
\centering
\includegraphics[width=\linewidth]{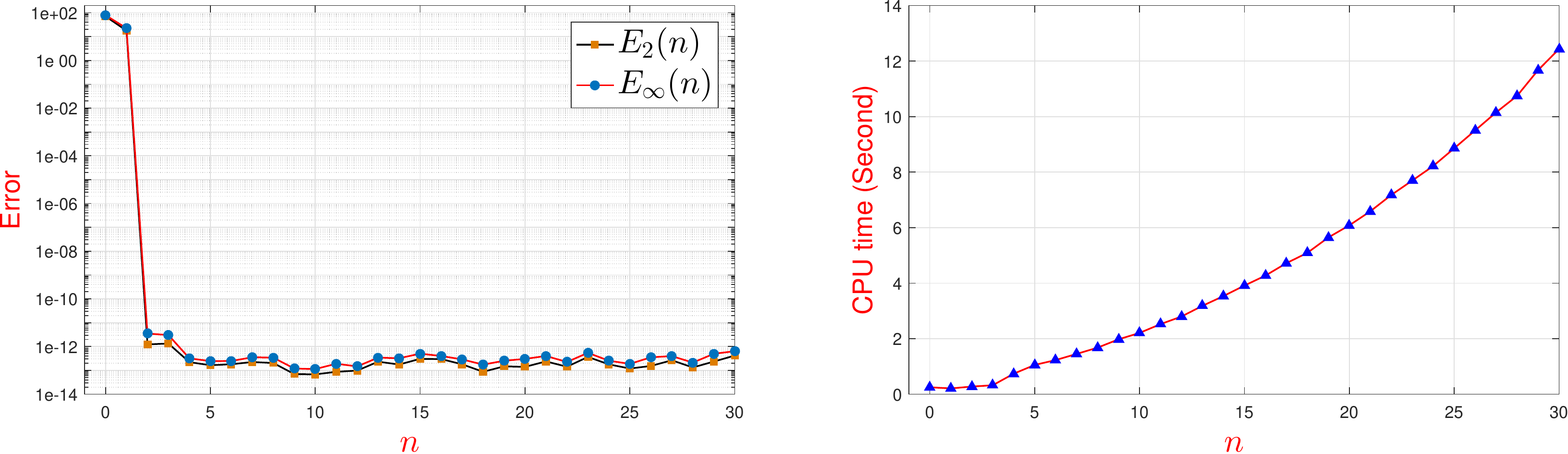}
\caption{Errors $E_2(n)$ and $E_\infty(n)$ for Example~1 (left) 
for various values of $n$ together with their CPU times (right).}
\label{fig:fig2ex1}
\end{figure}


\subsection*{Example 2}

For the second example, we consider problem \eqref{main_prob} 
with the following data:
\begin{align*}
&c_\alpha^+(x,t)=c_\alpha^-(x,t):=\frac {1}{\cos( \alpha\pi/2)},
\quad c_\beta^+(x,t)=c_\beta^-(x,t)
:=\frac {-1} {\cos( \beta\pi/2)},\quad \ell=:1,\quad f(x):=0,\\
&s(x,t):={\frac {24\,{t}^{\gamma}{{\rm e}^{\alpha\,t}}}{\cos\left(\alpha
\pi/2\right)} \left( {\frac {{x}^{2-\alpha}+ \left( 1-x \right) ^{
2-\alpha}}{12\,\Gamma  \left( 3-\alpha \right) }}-{\frac {{x}^{3-
\alpha}+ \left( 1-x \right) ^{3-\alpha}}{2\,\Gamma  \left( 4-\alpha
\right) }}+{\frac {{x}^{4-\alpha}+ \left( 1-x \right) ^{4-\alpha}}{
\Gamma  \left( 5-\alpha \right) }} \right) }\\
&\qquad \qquad\qquad + {\frac {24\,{t}^{\gamma}{{\rm e}^{\alpha\,t}}}{
\cos \left( \beta\pi/2  \right) }\left(\,{\frac {{x}^{2-\beta}
+ \left( 1-x \right) ^{2-\beta}}{12\,\Gamma\left( 3-\beta \right)}}
-{\frac {{x}^{3-\beta}+ \left( 1-x\right) ^{3-\beta}}{2\,
\Gamma\left( 4-\beta \right) }}+{\frac {{x}^{4-\beta}
+ \left( 1-x \right) ^{4-\beta}}{\Gamma  \left( 5-\beta\right) }} \right)}\\
&\qquad\qquad\qquad \qquad + {t}^{\gamma-1}{{\rm e}^{\alpha\,t}} 
\left( \alpha\,t+\beta \right) {x}^{2}\left(1-x\right)^{2}.
\end{align*}
The exact solution for this problem is
\begin{equation*}
u_{\text{ex}}(x,t)
:={t}^{\gamma}{{\rm e}^{\alpha\,t}}{x}^{2}\left( 1-x \right) ^{2}.
\end{equation*}
This example is taken from Li et al\cite{li2018fractional}.
In Table~\ref{table:01}, we report the obtained results by applying 
the proposed method with $\gamma=2$ and some 
values of $\alpha$ and $\beta$. As it is seen, accurate 
results are obtained even by using a small $n$.
\begin{table}
\caption{The $l_2$ and $l_\infty$ errors for Example~2,
obtained by the proposed method 
with $n=1,2,3$ and some values of $\alpha$ and $\beta$ 
with $\gamma=2$.\label{table:01}}
\begin{tabular}{lllllllllllllllllll} \hline
&  &  &  & \multicolumn{3}{c}{$\beta=1.2$} &  
& \multicolumn{3}{c}{$\beta=1.4$} &  & \multicolumn{3}{c}{$\beta=1.6$} 
&  & \multicolumn{3}{c}{$\beta=1.8$} \\ 
\cline{1-3}\cline{5-7}\cline{9-11}\cline{13-15}\cline{17-19}
$\alpha$ &  & $n$ &  & $E_2(n)$ &  & $E_\infty(n)$ &  
& $E_2(n)$ &  & $E_\infty(n)$ &  & $E_2(n)$ &  
& $E_\infty(n)$ &  & $E_2(n)$ &  & $E_\infty(n)$ \\ 
\cline{1-1}\cline{3-3}\cline{5-5}\cline{7-7}\cline{9-9}
\cline{11-11}\cline{13-13}\cline{15-15}\cline{17-17}\cline{19-19}
0.20 &&  1 && 5.5e-03 && 4.6e-01 && 5.2e-03 && 4.6e-01 && 5.0e-03 && 4.6e-01 && 4.6e-03 && 4.5e-01 \\ 
0.20 &&  2 && 8.5e-14 && 1.4e-11 && 4.4e-14 && 7.9e-12 && 5.6e-14 && 1.1e-11 && 1.4e-14 && 3.0e-12 \\ 
0.20 &&  3 && 8.3e-14 && 1.4e-11 && 4.3e-14 && 7.8e-12 && 5.6e-14 && 1.1e-11 && 1.4e-14 && 3.0e-12 \\[4pt] 
0.40 &&  1 && 6.4e-03 && 5.6e-01 && 6.2e-03 && 5.6e-01 && 5.9e-03 && 5.6e-01 && 5.5e-03 && 5.4e-01 \\ 
0.40 &&  2 && 8.2e-14 && 1.4e-11 && 4.3e-14 && 8.0e-12 && 5.5e-14 && 1.1e-11 && 1.3e-14 && 3.0e-12 \\ 
0.40 &&  3 && 8.1e-14 && 1.4e-11 && 4.2e-14 && 7.8e-12 && 5.5e-14 && 1.1e-11 && 1.3e-14 && 3.0e-12 \\[4pt]
0.60 &&  1 && 7.6e-03 && 6.8e-01 && 7.3e-03 && 6.9e-01 && 7.0e-03 && 6.8e-01 && 6.6e-03 && 6.7e-01 \\ 
0.60 &&  2 && 7.8e-14 && 1.3e-11 && 4.1e-14 && 7.6e-12 && 5.4e-14 && 1.1e-11 && 1.3e-14 && 3.0e-12 \\ 
0.60 &&  3 && 7.9e-14 && 1.4e-11 && 4.1e-14 && 7.9e-12 && 5.3e-14 && 1.1e-11 && 1.3e-14 && 2.9e-12 \\[4pt]
0.80 &&  1 && 9.0e-03 && 8.4e-01 && 8.7e-03 && 8.4e-01 && 8.3e-03 && 8.4e-01 && 7.8e-03 && 8.2e-01 \\ 
0.80 &&  2 && 7.4e-14 && 1.3e-11 && 4.0e-14 && 7.7e-12 && 5.2e-14 && 1.1e-11 && 1.3e-14 && 2.9e-12 \\ 
0.80 &&  3 && 7.4e-14 && 1.3e-11 && 4.0e-14 && 7.7e-12 && 5.2e-14 && 1.1e-11 && 1.3e-14 && 2.9e-12 \\ \hline
\end{tabular}
\end{table}


\subsection*{Example 3}

In this example, we consider problem \eqref{main_prob} with the data 
\begin{align*}
&c_\alpha^+=c_\alpha^-:=\frac{K_\alpha}{2\cos(\tfrac \pi 2 \alpha)},
\quad c_\beta^+=c_\beta^-:=\frac{-K_\beta}{2\cos(\tfrac \pi 2 \beta)},
\quad \ell:=\pi,\quad T:=4, \quad f(x):=\sin(4x).
\end{align*}
This example, with $s(x,t):=0$ and different values 
of parameters $K_\alpha$ and $K_\beta$, 
is treated in Le et al\cite{li2018fractional} and Yang et al\cite{QYang}.
When $s(x,t)=0$, the exact solution to this problem is, however, not known. 
For $\alpha=0.5$, $\beta=1.5$ and
\begin{align*}
s(x,t)&:=2\,{{\rm e}^{-t}} \Bigg\{ {\frac {2 K_\beta}{\sqrt {2\,\pi-2\,x}
\sqrt {\pi}}}-{\frac {\sqrt {2}K_\beta}{\sqrt {x\pi}}}
-\frac{\sin\left( 4\,x \right)} 2\\
&+ \big[ K_\alpha\sin \left( 4x\right) -4K_\beta\cos \left( 4x \right)  \big] 
\texttt{FresnelS} \left( \sqrt{\tfrac {{8}}{ {\pi}}x} \right) 
+ \big[ K_\alpha\sin \left( 4 x \right) 
+4K_\beta\cos\left( 4 x \right)\big] \texttt{FresnelS} \left( 
\sqrt{{{8-\tfrac 8 \pi x}}} \right)\\
&+ \big[ K_\alpha\cos \left( 4x \right) 
+4 K_\beta\sin \left( 4 x\right)  \big] \texttt{FresnelC} 
\left(\sqrt{\tfrac 8 \pi x} \right) - \big[ K_\alpha\cos \left( 4x
\right) -4 K_\beta\sin \left( 4 x \right)  \big] 
\texttt{FresnelC} \left(  \sqrt{8-\tfrac 8 \pi x} \right)  \bigg\},
\end{align*}
the exact solution of this problem is given by
\[
u_{\text{ex}}(x,t):={{\rm e}^{-t}}\sin(4x).
\]
We point out that \texttt{FresnelC} and \texttt{FresnelS} are the 
Fresnel cosine and Fresnel sine integral functions, respectively, 
which are defined as
\begin{equation*}
{\texttt{ FresnelC}}(x):=\int_0^x \cos\left( 
\frac {\pi \mu^2}{2}\right)\text{d}\mu,\qquad
{\texttt{FresnelS}}(x):=\int_0^x \sin\left( 
\frac {\pi \mu^2}{2}\right)\text{d}\mu.
\end{equation*}
In contrast with the last two examples, here the exact solution is not 
a polynomial type function and thus this example is more proper for assessing 
the accuracy and convergence rate of our method. With this purpose, 
the errors $E_2(n)$ and $E_\infty(n)$ for various values of $n$ with 
$K_\alpha=2$ and $K_\beta=0.1$  are plotted in Fig.~\ref{fig:fig3ex3} (left). 
It can be observed from this figure that the errors decrease  
rapidly and the spectral convergence rate is gained by the method.  
Moreover, the obtained solution with $n=30$ and its corresponding 
error are plotted in Fig.~\ref{fig:fig3ex3} (right).

In addition, in order to show the effects of parameters $\alpha$, $\beta$, 
$K_\alpha$ and $K_\beta$, the obtained solutions for Example~3 with 
$s(x,t)=0$ are plotted in Fig.~\ref{fig:fig4ex3}
for the parameter values of Table~\ref{table:02}. 
\begin{figure}
\centering
\includegraphics[width=\linewidth]{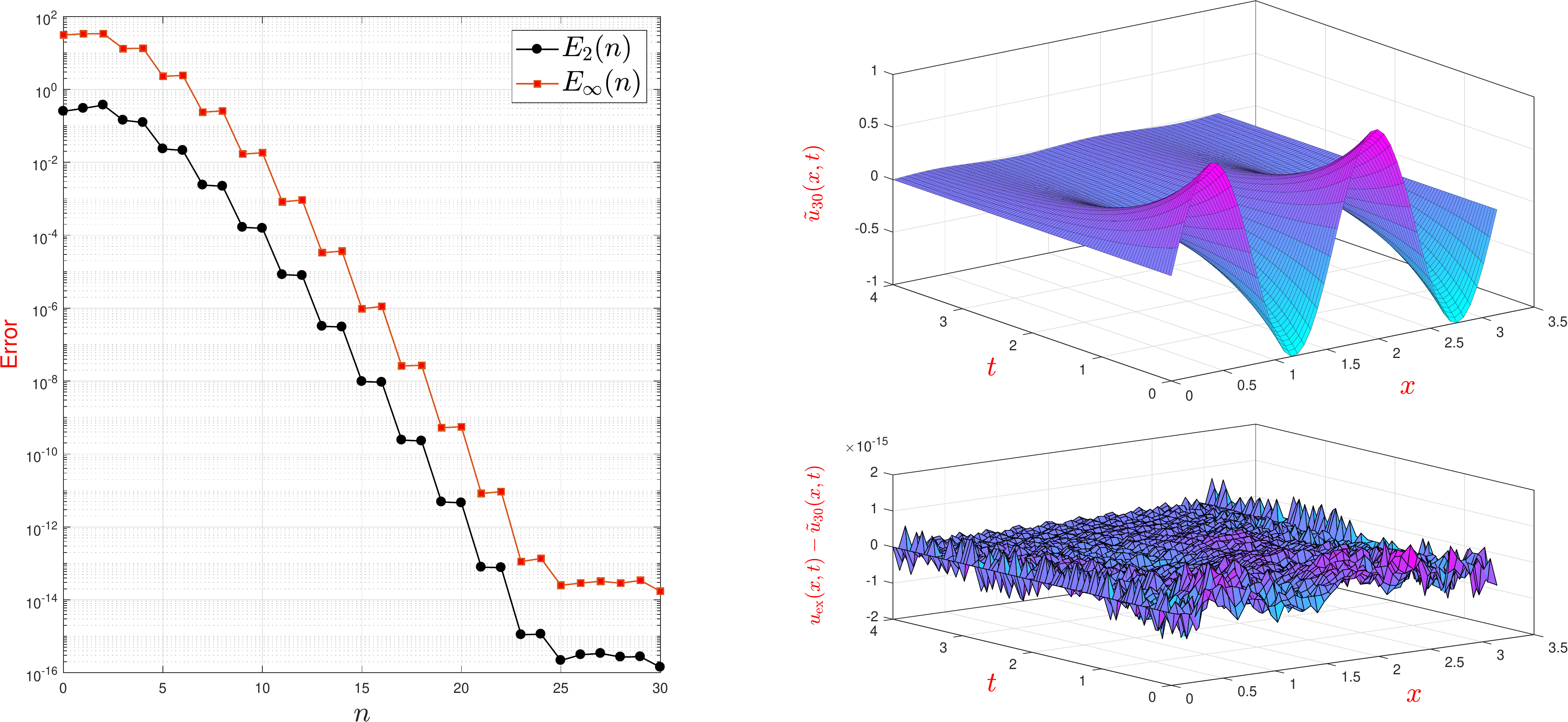}
\caption{Example 3, with $\alpha=0.5$, $\beta=1.5$, $K_\alpha=2$ 
and $K_\beta=0.1$. Left: $E_2(n)$ and $E_\infty(n)$  for various values of $n$. 
Right: the obtained solution with $n=30$ and its corresponding error.}
\label{fig:fig3ex3}
\end{figure}
\begin{figure}
\centering
\includegraphics[width=\linewidth]{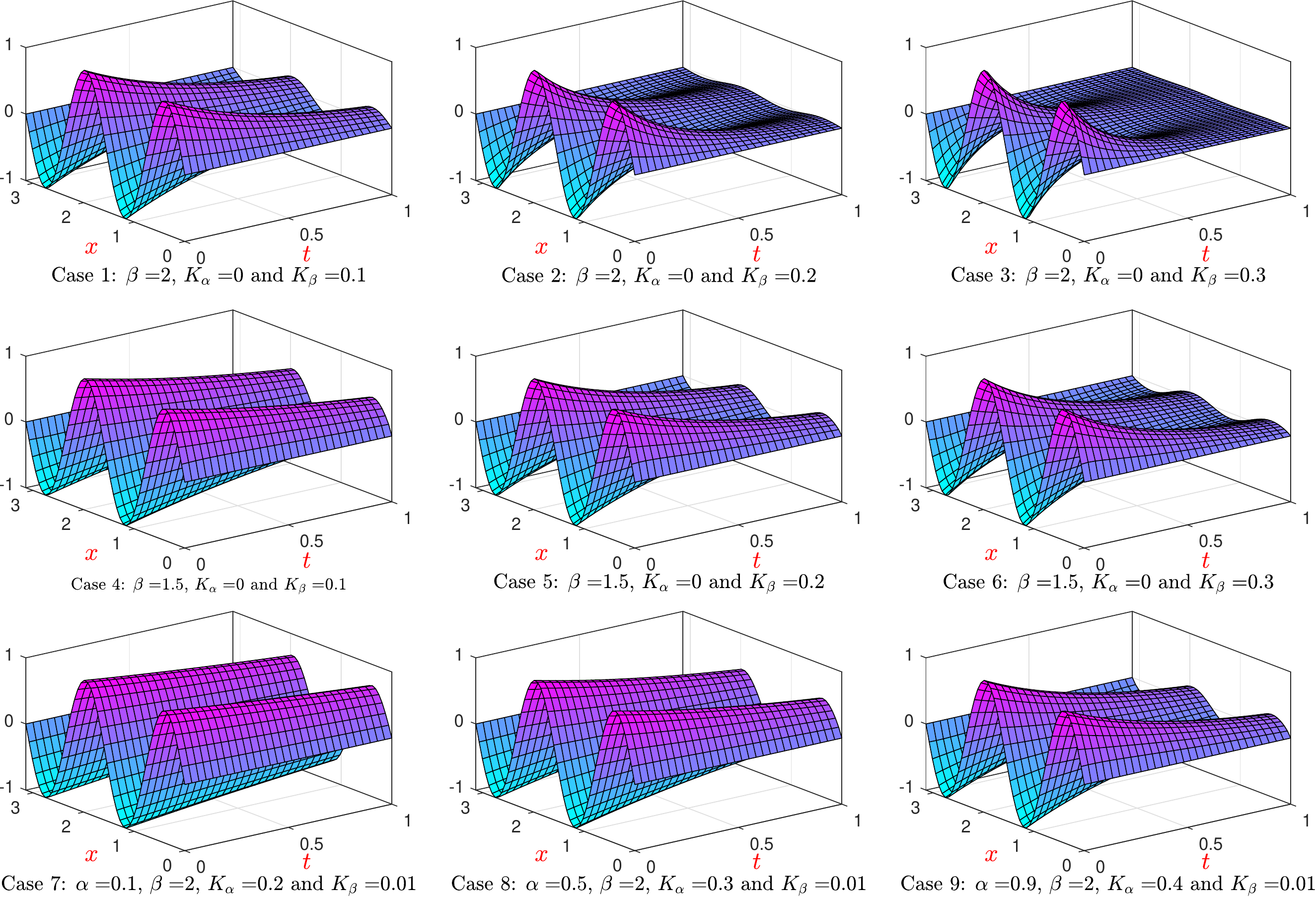}
\caption{Obtained numerical solutions to Example~3 with $s(x,t)=0$
and various values of $\alpha$, $\beta$, $K_\alpha$ and $K_\beta$, for $n=50$.}
\label{fig:fig4ex3}
\end{figure}
\begin{table}
\caption{Different values of parameters $\alpha$, $\beta$, 
$K_\alpha$ and $K_\beta$ considered in Example~3 for problem
with $s(x,t)=0$ (see Fig.~\ref{fig:fig4ex3}).\label{table:02}}
\begin{center}
\begin{tabular}{llllllllll} \hline
& Case 1 & Case 2 & Case 3 & Case 4 & Case 5 
& Case 6 & Case 7 & Case 8 & Case 9 \\ \hline
$K_\alpha$ & \multicolumn{1}{c}{0.0} & \multicolumn{1}{c}{0.0} 
& \multicolumn{1}{c}{0.0} & \multicolumn{1}{c}{0.0} 
& \multicolumn{1}{c}{0.0} & \multicolumn{1}{c}{0.0} 
& \multicolumn{1}{c}{0.2} & \multicolumn{1}{c}{0.3} & \multicolumn{1}{c}{0.4}\\
$\alpha$ & \multicolumn{1}{c}{-} & \multicolumn{1}{c}{-} 
& \multicolumn{1}{c}{-} & \multicolumn{1}{c}{-} & \multicolumn{1}{c}{-} 
& \multicolumn{1}{c}{-} & \multicolumn{1}{c}{0.1} & \multicolumn{1}{c}{0.5} 
& \multicolumn{1}{c}{0.9} \\ 
$K_\beta$ & \multicolumn{1}{c}{0.1} & \multicolumn{1}{c}{0.2} 
& \multicolumn{1}{c}{0.3} & \multicolumn{1}{c}{0.1} 
& \multicolumn{1}{c}{0.2} & \multicolumn{1}{c}{0.3} 
& \multicolumn{1}{c}{0.01} & \multicolumn{1}{c}{0.01} & \multicolumn{1}{c}{0.01} \\ 
$\beta$ & \multicolumn{1}{c}{2.0} & \multicolumn{1}{c}{2.0} 
& \multicolumn{1}{c}{2.0} & \multicolumn{1}{c}{1.5} 
& \multicolumn{1}{c}{1.5} & \multicolumn{1}{c}{1.5} 
& \multicolumn{1}{c}{2.0} & \multicolumn{1}{c}{2.0} & \multicolumn{1}{c}{2.0} \\ \hline
\end{tabular}
\end{center}
\end{table}


\subsection*{Example 4}

As a last example, we consider problem \eqref{main_prob} 
with discontinuous data. Let
\begin{align*}
&\ell:=7,\quad T:=1,\quad s(x,t):=0,
\quad c_\alpha^+=c_\alpha^-:=1,\quad 
f(x):=\begin{cases}
1,& 1\le x <2,\\
2,& 3\le x<4,\\
4,& 5\le x <6,\\
0,& \text{otherwise},
\end{cases}
\end{align*}
and consider the diffusion coefficients in the following two cases:
\begin{equation}
\label{eq:ex4:cases}
\text{ Case I: }c_\beta^+=c_\beta^-
:=\begin{cases}
0.1,& 0\le x <4.5,\\
0.001, & 4.5\le x \le 7,
\end{cases}
\quad 
\text{ Case II: }c_\beta^+=c_\beta^-
:=\begin{cases}
0,& 0\le x <4.5,\\
0.7, & 4.5\le x \le 7.
\end{cases}
\end{equation}
By applying the proposed method with $n=200$ to the problem 
in Case I and Case II, the solutions are obtained  after 
3.1 and 3.2 seconds, respectively. These solutions are plotted 
in Fig.~\ref{fig:figex4}. As we can see, the obtained solutions 
are confirmed from a physical point of view. We conclude
that the presented method can be applied with success 
to challenging problems with discontinuous data.
\begin{figure}
\centering
\includegraphics[width=\linewidth]{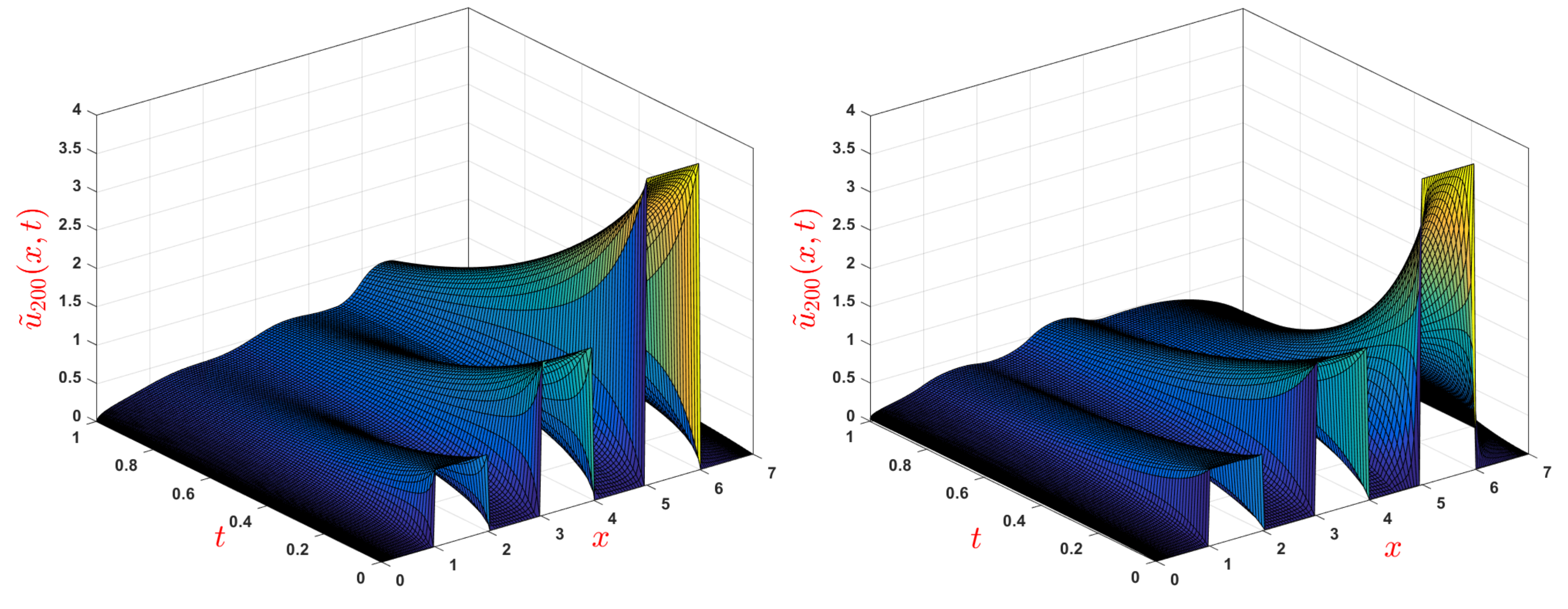}
\caption{Approximate solutions obtained by the proposed method
with $n=200$ for Example~4 in Cases I (left) 
and II (right) given by \eqref{eq:ex4:cases}.}
\label{fig:figex4}
\end{figure}


\section{Conclusion}
\label{sec:05:conc}

We proposed an efficient and accurate numerical method 
for solving two-sided space-fractional advection-diffusion equations. 
We started by introducing a new set of basis functions,  
called the Modified Jacobi Functions (MJFs). The 
method of lines, together with the spectral collocation method based 
on the introduced basis functions, were successfully 
applied to the considered problems. In order to increase 
the efficiency of the proposed method, from the numerical point of view,   
the left- and right-sided Riemann--Liouville fractional differentiation 
matrices were derived. Several numerical examples were provided, 
showing good efficiency and high accuracy ({\it exponential accuracy}). 
The obtained results confirm that our method 
not only works well for problems with smooth solutions 
(see Examples~1, 2 and 3) but also can be easily applied  
for problems with discontinuities on the initial condition, 
advection or diffusion coefficients (see Example~4).   
Obtaining some theoretical estimates for the approximation 
errors and convergence analysis would be desirable. 
This work is currently in progress. A further interesting topic for research 
is using MJFs, with $\rho,\,\theta \in (0,1)$, as the basis functions 
to solve the considered problem, by a Galerkin scheme. 


\section*{Acknowledgments}

Torres was partially supported by FCT through the R\&D Unit CIDMA, 
project UID/MAT/04106/2019. The authors are very grateful to anonymous 
referees for carefully reading their manuscript and for several 
comments and suggestions, which helped them to improve the paper.



\end{document}